\documentclass[12pt]{article}
\usepackage{wrapfig}

\usepackage[top=3cm,left=3cm,right=3cm,bottom=3cm]{geometry}
\usepackage{hyperref}
\usepackage{amssymb,latexsym,amsmath,epsfig,amsthm, tikz, verbatim, hyperref,color} 
\usepackage[english]{babel}
\usepackage{tikz-cd} 
\usepackage[section]{placeins}
\makeatletter

\renewcommand\section{\@startsection {section}{1}{\z@}
{-30pt \@plus -1ex \@minus -.2ex}
{2.3ex \@plus.2ex}
{\normalfont\normalsize\bfseries}}

\renewcommand\subsection{\@startsection{subsection}{2}{\z@}
{-3.25ex\@plus -1ex \@minus -.2ex}
{1.5ex \@plus .2ex}
{\normalfont\normalsize\bfseries}}

\renewcommand{\@seccntformat}[1]{\csname the#1\endcsname. }

\makeatother

\newtheorem{theorem}{Theorem}
\numberwithin{theorem}{section}
\newtheorem{proposition}[theorem]{Proposition}
\newtheorem{lemma}[theorem]{Lemma}
\newtheorem{corollary}[theorem]{Corollary}

\theoremstyle{definition}
\newtheorem{definition}[theorem]{Definition}
\newtheorem{remark}[theorem]{Remark}
\newtheorem{example}[theorem]{Example}

\begin{document}

\title{Offset Hypersurfaces and Persistent Homology of Algebraic Varieties}

\author{Emil Horobe\c{t} and Madeleine Weinstein}
\date{}
\maketitle

\begin{abstract}
In this paper, we study the persistent homology of the offset filtration of algebraic varieties. We prove the algebraicity of two quantities central to the computation of persistent homology. Moreover, we connect persistent homology and algebraic optimization.  Namely, we express the degree corresponding to the distance variable of the offset hypersurface in terms of the Euclidean Distance Degree of the starting variety, obtaining a new way to compute these degrees. Finally, we describe the non-properness locus of the offset construction and use this to describe the set of points that are topologically interesting (the medial axis and center points of the bounded components of the complement of the variety) and relevant to the computation of persistent homology.  
\end{abstract}

\section{Introduction}

Experimental research is based on collecting and analyzing data. It is very important to understand the background mathematical model that defines a given phenomenon. One of the possibilities is that the data is driven by a geometric model, say an algebraic variety or a manifold. In this case, we would like to ``learn the geometric object" from the data (for more details see \cite{BKSW18}). For example, we would like to understand the topological features of the underlying model. A common way to do this is by \textit{persistent homology} (\cite{Carlsson, NSW, CZ}), 
which studies the homology of the set of points within a range of distances from the data set, and considers features to be of interest if they persist through a wide range of the distance parameter.

This article is at the intersection of computational geometry, geometric design, topology and algebraic geometry, linking all of these topics together. In what follows, we study the \textit{persistent homology of the offset filtration} of an algebraic variety, which we define to be the homology of its offsets. Related work includes \cite{HKS15} in which the notion of persistent homology is extended to the offsets of convex objects. 

We show that the indicators (barcodes) of the persistent homology of the offset filtration of a variety defined over the rational numbers are algebraic and thus can be computed exactly (Theorem~\ref{algebraicity}). Moreover, we connect persistent homology and algebraic optimization (Euclidean Distance Degree~\cite{DHOST16}) through the theory of offsets, bringing insights from each field to the other.  Namely, we express the degree corresponding to the distance variable of the offset hypersurface in terms of the Euclidean Distance Degree of the original variety (Theorem~\ref{Eps_Degree}), obtaining a new way to compute these degrees. A consequence of this result is a bound on the degree of the \textit{ED discriminant} (Corollary~\ref{degree_discriminant}) and on the degree of the closure of the medial axis (see \ref{medial_axis}). We describe the non-properness locus of the offset construction (Subsection $2.1$) and use this to describe the set of points (Theorem~\ref{interesting_points}) in the ambient space that are topologically interesting (the medial axis and center points of the bounded components of the complement of the variety) and relevant to the computation of persistent homology. Lastly, we show that the reach of a manifold, the quantity used to ensure the correctness of persistent homology computations, is algebraic (Proposition \ref{reach}).   

The article is structured as follows. Section~\ref{Sec2} discusses offset hypersurfaces. We analyze the construction, dimension and degree of offsets and define the offset discriminant. Section~\ref{Sec3} is about persistent homology. We review background material on persistent homology, define the persistent homology of the offset filtration of an algebraic variety and prove its algebraicity, connect the offset discriminant to topologically interesting points in the complement of the variety, and prove the algebraicity of the reach.

\section{Offset hypersurfaces of algebraic varieties}\label{Sec2}

We devote this section to the algebraic study of offset hypersurfaces. Driven by real world applications, our starting variety $X_{\mathbb{R}}\subseteq \mathbb{R}^n$ is a real irreducible variety and we construct its $\epsilon$-offset hypersurface, for any generic real positive $\epsilon$. In order to use techniques from algebraic geometry, we consider the variety $X\subseteq \mathbb{C}^n$ that is the complexification of $X_{\mathbb{R}}$ and let $\epsilon$ be any complex number. In what follows, by the squared distance of two points $x,y\in \mathbb{C}^n$ we will mean the complex value of the function $d(x,y)=\sum_{i=1}^n(x_i-y_i)^2$. This is not the usual Hermitian distance function on $\mathbb{C}^n$, but rather the complexification of the real Euclidean distance function. It is not a metric on $\mathbb{C}^n$, but it is a metric when restricted to $\mathbb{R}^n$.

\subsection{Offset construction}

Let $X\subseteq \mathbb{C}^n$ be an irreducible variety of codimension $c$ and let $\epsilon$ be a fixed (generic) complex number. By an $\epsilon$-hyperball centered at a point $y\in\mathbb{C}^n$, we mean the variety $V(d(x,y)-\epsilon^2)$.

\begin{definition}
The \textbf{$\epsilon$-offset hypersurface} is defined to be the union of the centers of $\epsilon$-hyperballs that intersect the variety $X$ non-transversally at some point $x\in X$. Equivalently the $\epsilon$-offset hypersurface is the envelope of the family of $\epsilon$-hyperballs centered on the variety. For a fixed $\epsilon$ we denote the $\epsilon$-offset hypersurface by $\mathcal{O}_{\epsilon}(X)$.
\end{definition}

\begin{figure}[h]
\begin{center}
\vskip -0.2cm
\includegraphics[scale=1.1]{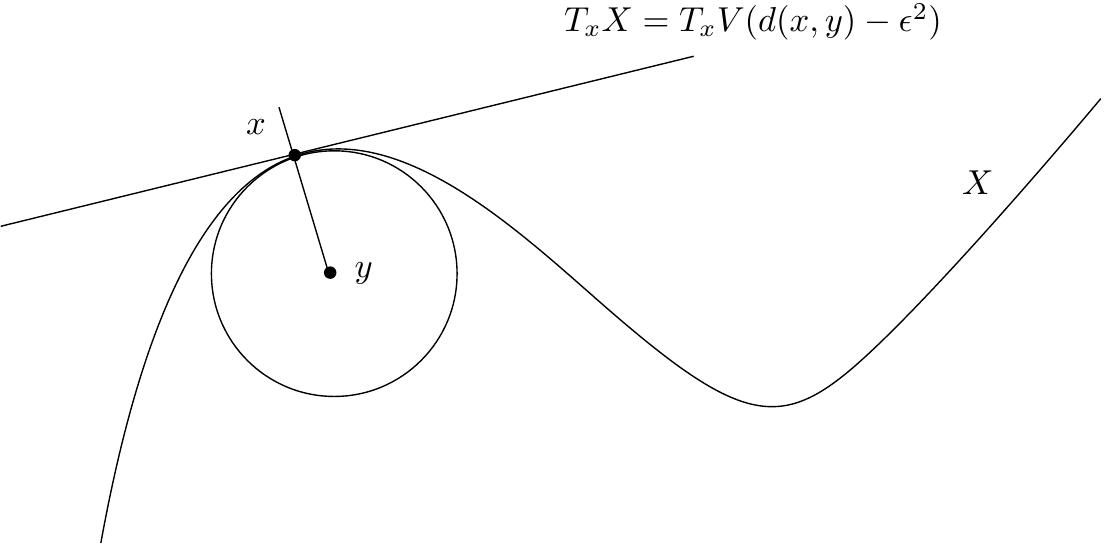}
\vskip -0.2cm
\caption{Non-transversal intersection of the variety with the $\epsilon$-hyperball.}
\label{Offset hypersurface construction}
\end{center}
\end{figure}

Let $y\in\mathcal{O}_{\epsilon}(X)$, the above-defined $\epsilon$-offset hypersurface. Then there exists an $x\in X_{reg}$, that is a regular (nonsingular) point of the variety, such that the squared distance $d(x,y)$ is exactly $\epsilon^2$ and by the non-transversality $T_x X\subseteq T_x V(d(x,y)-\epsilon^2)$. Hence
\[
x-y\perp T_x X,
\] where $T_x X$ is the tangent space at $x$ to $X$ and $T_x V(d(x,y)-\epsilon^2)$ is the tangent space at $x$ to $V(d(x,y)-\epsilon^2$), the variety defined  by the vanishing of the polynomial $d(x,y)-\epsilon^2$, which is the $\epsilon$-hyperball centered at $y$.

The latter condition can be described by polynomial equations as follows (see for example \cite[Section $2$]{DHOST16}). The condition $x-y \perp T_x X$ is satisfied if and only if the rank of

\[\left(
\begin{array}{cc}
x-y \\
\mathrm{Jac}_x (I)\\
\end{array}\right)
\] 
is less then $c+1$, where $\mathrm{Jac}_x (X)$ is the Jacobian of the defining radical ideal of the variety $X$, at the point $x$ (the matrix of all the partial derivatives of all the minimally defining polynomials of $X$). Namely $x-y \perp T_x X$ if and only if all the $(c+1)\times (c+1)$ minors of the matrix above vanish.

To capture the entire geometry behind the construction of the offset hypersurface we consider the closure of the set of all pairs $(x,y)\in \mathbb{C}^n\times \mathbb{C}^n$ such that $x\in X_{reg}$ and $y$ satisfies the conditions above. We name this variety the \textbf{offset correspondence} of $X$ and denote it by $\mathcal{OC}_{\epsilon}(X)$. This correspondence is a variety in $\mathbb{C}^n_x\times \mathbb{C}^n_y$ and is equal to the closure of the intersection
\[\label{offset_equ}
(X_{reg}\times \mathbb{C}^n)\cap V\left( (c+1)\times (c+1) \text{ minors of } \left(\begin{array}{cc}
x-y \\
\mathrm{Jac}_x (I) \\
\end{array}\right) \right)\cap V( d(x,y)-\epsilon^2).
\]

Observe that the intersection of the first two varieties is the \textit{Euclidean Distance Degree correspondence}, $\mathcal{E}(X)$, that is the closure of the pair of points $(x,y)$ in $\mathbb{C}^n_x\times \mathbb{C}^n_y$, such that $x\in X_{reg}$ and $x-y\perp T_x X$. This correspondence contains pairs of ``data points" $y\in \mathbb{C}^n_y$ and corresponding points on the variety $x\in X_{reg}$, such that $x$ is a constrained critical point of the Euclidean distance function $d_{y}(x)=d(x,y)$ with respect to the constraint that $x\in X_{reg}$. For more details on this problem we direct the reader to \cite[Section 2]{DHOST16}.
Using the terminology of the Euclidean Distance Degree problem, we have
\begin{equation}\label{ED_corr_def}
\mathcal{OC}_{\epsilon}(X)=\mathcal{E}(X)\cap V(d(x,y)-\epsilon^2).
\end{equation}

From the offset correspondence, we have the natural projections $\mathrm{pr}_x: \mathcal{OC}_{\epsilon}(X) \to  \mathbb{C}^n_x$ and $\mathrm{pr}_y: \mathcal{OC}_{\epsilon}(X) \to  \mathbb{C}^n_y.$ The closure of the first projection is the variety $X$ and the closure of the second projection is the offset hypersurface $\mathcal{O}_{\epsilon}(X)$.

It follows that the \textbf{offset hypersurface} is
\[
\mathcal{O}_{\epsilon}(X)=\overline{\mathrm{pr}_y(\mathcal{OC}_{\epsilon}(X))}\subseteq \mathbb{C}^n_y.
\]

\begin{remark}\label{same_field}
When $X$ is a real variety, note that by the Tarski-Seidenberg Theorem (see Lemma \ref{TST}) the offset hypersurface is defined over the same closed real (sub)field as $X$ and $\epsilon$ are defined.
\end{remark}

In the following example, we illustrate an algorithm to compute the defining polynomial of the offset hypersurface of an ellipse using {\tt Macaulay2} \cite{M2}.
\begin{example}[\textbf{Computing the offset hypersurface of the ellipse}]\label{ellipse_offset_exa}
Consider the ellipse $X\subseteq \mathbb{C}^{2}$ defined by the vanishing of the polynomial $f=x_1^2+4x_2^2-4$. The code below outputs the defining ideal of the offset hypersurface in terms of the parameter $\epsilon$. 
\begin{verbatim}
n=2;
kk=QQ[x_1..x_n,y_1..y_n,e];
f=x_1^2+4*x_2^2-4;
I=ideal(f);
c=codim I;
Y=matrix{{x_1..x_n}}-matrix{{y_1..y_n}};
Jac= jacobian gens I;
S=submatrix(Jac,{0..n-1},{0..numgens(I)-1});
Jbar=S|transpose(Y);
EX = I + minors(c+1,Jbar);
SingX=I+minors(c,Jac);
EXreg=saturate(EX,SingX);
distance=Y*transpose(Y)-e^2;
Offset_Correspondence=EXreg+ideal(distance);
Off_hypersurface=eliminate(Offset_Correspondence,toList(x_1..x_n))
\end{verbatim}
\newpage
The result is that $\mathcal{O}_{\epsilon}(X)$ is the zero locus of the polynomial 
\[
y_1^8+10y_1^6y_2^2+33y_1^4y_2^4+40y_1^2y_2^6+16y_2^8+4y_1^6\epsilon^2-30y_1^4y_2^2\epsilon^2-90y_1^2y_2^4\epsilon^2
\]
\[
-56y_2^6\epsilon^2-2y_1^4\epsilon^4+62y_1^2y_2^2\epsilon^4+73y_2^4\epsilon^4-12y_1^2\epsilon^6-42y_2^2\epsilon^6+9\epsilon^8-14y_1^6
\]
\[
-90y_1^4y_2^2-120y_1^2y_2^4+64y_2^6-62y_1^4\epsilon^2+140y_1^2y_2^2\epsilon^2-248y_2^4\epsilon^2-90y_1^2\epsilon^4
\]
\[
+270y_2^2\epsilon^4-90\epsilon^6+73y_1^4+248y_1^2y_2^2-32y_2^4+270y_1^2\epsilon^2-360y_2^2\epsilon^2
\]
\[
+297\epsilon^4-168y_1^2-192y_2^2-360\epsilon^2+144.
\]

The code above is designed to work in arbitrary dimensions and for any variety. For this reason, we saturate by the singular locus of the variety, even though this step is unnecessary in this example as the ellipse is smooth. 
\end{example}

\begin{example}[\textbf{Offset hypersurface of a space curve}]
Let the variety $X$ be the Viviani curve in $\mathbb{C}^3$, defined by the intersection of a sphere with a cylinder tangent to the sphere and passing through the center of the sphere. So $X$ is defined by the vanishing of $f_1=x_1^2+x_2^2+x_3^2-4$ and $f_2=(x_1-1)^2+x_2^2-1$. In Figure~\ref{Viviani} the reader can see (on the left) the real part of the Viviani curve and (on the right) the $\epsilon=1$ offset surface of the curve. This surface is defined by a degree $10$ irreducible polynomial consisting of $175$ monomials.
\begin{figure}[h]
\begin{center}
\vskip -0.2cm
\includegraphics[scale=0.3]{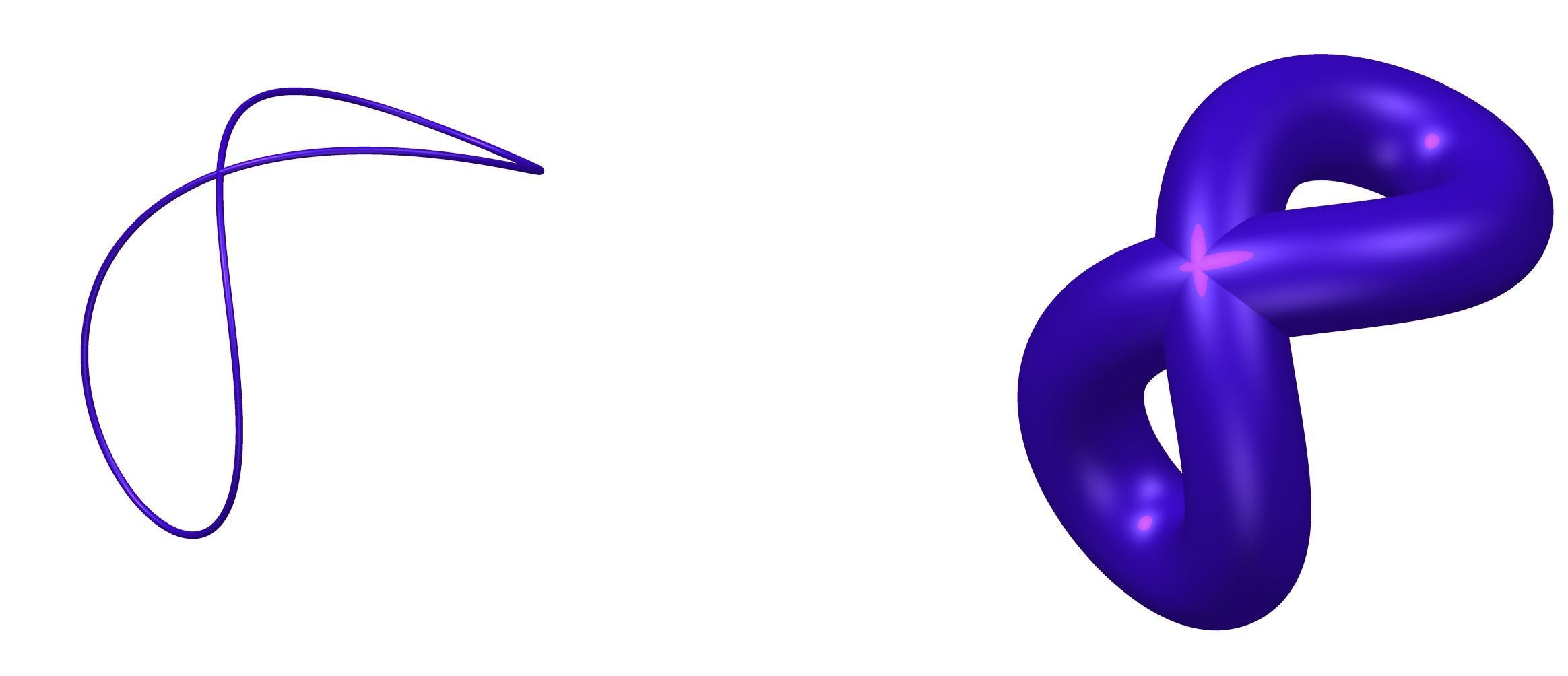}
\vskip -0.2cm
\caption{The Viviani curve (left) and its offset surface (right).}
\label{Viviani}
\end{center}
\end{figure}
\end{example} 
One could consider the family of all $\epsilon$-offset hypersurfaces $\mathcal{O}_{\epsilon}(X)\subseteq \mathbb{C}^n$ as $\epsilon$ varies over $\mathbb{C}$. This family is again a hypersurface in $\mathbb{C}^n_y\times \mathbb{C}_{\epsilon}^1$  defined by the same ideal as is $\mathcal{O}_{\epsilon}(X)$, but now $\epsilon$ is a variable.

Define the \textbf{offset family}, $\mathcal{O}(X)$, to be the closure of all offset hypersurfaces of $X$ in the $n+1$ dimensional space $\mathbb{C}^n_{y}\times \mathbb{C}^1_{\epsilon}$ defined by the same ideal as $\mathcal{O}_{\epsilon}(X)$. More precisely, let
\[
\mathcal{O}(X)=\overline{\{(y,\epsilon),\ y\in\mathcal{O}_{\epsilon}(X)\}}\subseteq \mathbb{C}^n_{y}\times \mathbb{C}^1_{\epsilon}.
\]

\begin{example}[\textbf{The offset family of an ellipse}]
A picture of the real part of $\mathcal{O}(X)$, where $X$ is the ellipse defined by the vanishing of $x_1^2+4x_2^2-4=0$, can be seen below in Figure~\ref{Offset_ellipse}. It is the set of all points $(y_1,y_2,\epsilon)\in \mathbb{C}^3$ that are zeros of the polynomial in \ref{ellipse_offset_exa}.

\begin{figure}[h]
\begin{center}
\vskip -0.2cm
\includegraphics[scale=0.2]{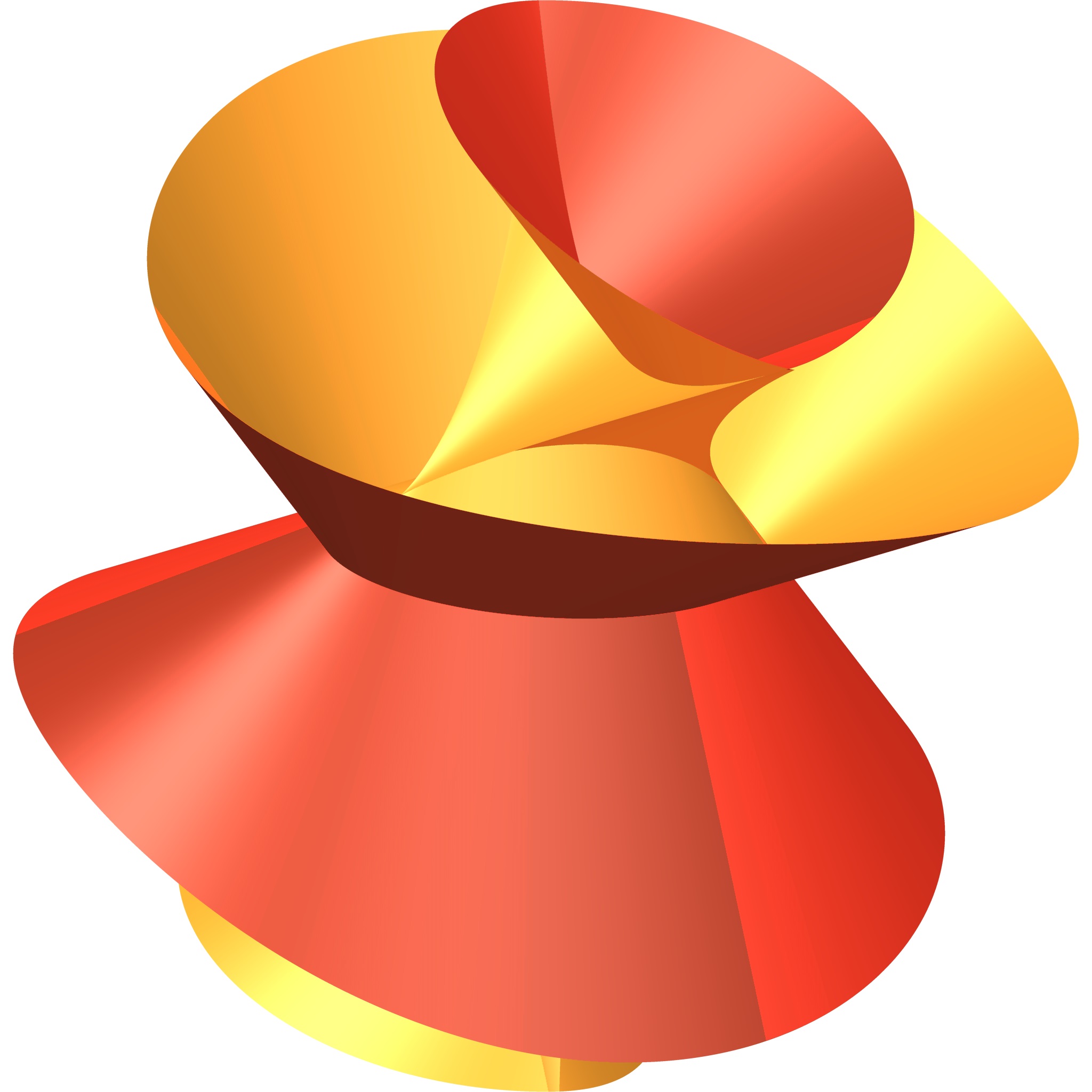}
\vskip -0.2cm
\caption{Offset family of an ellipse.}
\label{Offset_ellipse}
\end{center}
\end{figure}

A horizontal cut (by a plane $\epsilon=\epsilon_0$) of the surface above is the $\epsilon_0$-offset curve of the ellipse.
\end{example}

\subsection{Offset dimension and degree}

The degree and dimension are important invariants of an algebraic variety. These invariants of the offset hypersurface have been studied by many authors (for example by San Segundo and Sendra in \cite{SegundoSendra05,SegundoSendra09,SegundoSendra12}) in both the implicit and the parametric cases. To supplement the existing literature, in this subsection we relate the degree in $\epsilon$ (or $\epsilon$-degree) of the defining polynomial of generic offset hypersurfaces to the ED degree of the original variety, for any irreducible variety $X$. In this way, we achieve a new method for computing both ED degrees of varieties and degrees of offsets.

We now recall a theorem crucial in further understanding the essence of the offset construction.

\begin{theorem}[Theorem 4.1 from \cite{DHOST16}]

The Euclidean Distance Degree correspondence $\mathcal{E}(X)$ (see Equation \ref{ED_corr_def}) is an irreducible variety of dimension $n$ inside $\mathbb{C}_x^n \times \mathbb{C}^n_y$. 
The first projection $\mathrm{pr}_{x} : \mathcal{E}(X) \to X \subseteq \mathbb{C}_x^n$ is an affine vector bundle of rank $c$ over $X_{reg}$. 
Over generic $y_0 \in \mathbb{C}^n_y$, the second projection $\mathrm{pr}_y: \mathcal{E}(X) \to
\mathbb{C}^n_y$ has finite fibers $\mathrm{pr}_y^{-1}(y_0)$ of cardinality equal (by definition) to the \textit{Euclidean Distance Degree} (ED degree) of $X$.

\[\begin{tikzcd}
 &\mathcal{E}(X) \arrow{ld}[swap]{\mathrm{pr}_{x}}\arrow{rd}{\mathrm{pr}_y}& \\
 \mathbb{C}^n_x & & \mathbb{C}^n_y
\end{tikzcd}
\]
\end{theorem}

\begin{remark}\label{def_Ed_disc}
	The second projection, $\mathrm{pr}_y$, has a ramification locus which is generically a hypersurface in $\mathbb{C}^n_y$, by the Nagata-Zariski Purity Theorem \cite{Purity2},\cite{Purity1}. The \textit{Euclidean Distance discriminant (ED discriminant)} is the closure of the image of the ramification locus of $\mathrm{pr}_y$, (i.e. the points where the derivative of $\mathrm{pr}_y$ is not of full rank, under the projection $\mathrm{pr}_y$). As in~\cite[Section 7]{DHOST16}, we denote the ED discriminant of the variety $X$ by $\Sigma(X)$.
\end{remark}

The offset correspondence is the intersection of the Euclidean Distance Degree correspondence with the hypersurface $V(d(x,y)-\epsilon^2)$ in $\mathbb{C}^n_x\times \mathbb{C}^n_y$ (recall Equation \ref{ED_corr_def}). This intersection is $n-1$ dimensional because $\mathcal{E}(X)$ is not a  subvariety of $V(d(x,y)-\epsilon^2)$ (because not all pairs $(x,y)\in \mathcal{E}(X)$ are at $\epsilon^2$ squared distance from each other). 
As a consequence the offset correspondence, $\mathcal{OC}_{\epsilon}(X)$, is an $n-1$ dimensional variety in $\mathbb{C}^n_x\times \mathbb{C}^n_y$. But over generic $y_0 \in \mathbb{C}^n_y$, the projection \[\mathrm{pr}_y:\mathcal{OC}_{\epsilon}(X) \to \mathbb{C}^n_y\] has finite fibers, so the closure of the image, $\overline{\mathrm{pr}_y(\mathcal{OC}_{\epsilon}(X))}=\mathcal{O}_{\epsilon}(X)$ is $n-1$ dimensional as well, hence the name offset \textit{hypersurface}. For a more detailed analysis of the dimension degeneration of components of the offset hypersurface see~\cite{SendraSendra00}.

\begin{remark}\label{ED_many_times}
Observe that a fixed generic $y_0$ is an element of the offset hypersurface $\mathcal{O}_{\epsilon}(X)$ for precisely two times ED degree many distinct $\epsilon$.  This is because $y_0$ has ED degree many critical points to $X$, say $\{x_1,\ldots,x_{EDdegree(X)}\}$ and then the corresponding offset hypersurfaces that include $y_0$, are the ones where $\epsilon$ is in 
\[\left\{\pm \sqrt{d(x_1,y_0)},\ldots, \pm \sqrt{d(x_{EDdegree(X)},y_0)}\right\}.\]
\end{remark}

\begin{theorem}\label{Eps_Degree}
The degree in $\epsilon$ (or $\epsilon$-degree) of the defining polynomial of $\mathcal{O}(X)$ (the offset family) is equal to two times the Euclidean Distance degree of the variety $X$.
\end{theorem}
\begin{proof}
Suppose that $\mathcal{O}(X)$ is defined by $f(y,\epsilon)$. By Remark~\ref{ED_many_times}, a generic $y_0$ is an element of $\mathcal{O}_{\epsilon}(X)$ for precisely two times $ED$ degree many $\epsilon$. This is equivalent to $f(y_0,\epsilon)$ having exactly two times ED degree many roots. And these roots are
\[\left\{\pm \sqrt{d(x_1,y_0)},\ldots, \pm \sqrt{d(x_{EDdegree(X)},y_0)}\right\},\]
where $x_i$ are critical points of the distance from $y_0$ to the variety.
\end{proof}

We note that San Segundo and Sendra~\cite{SegundoSendra09} derived the $\epsilon$-degree of plane offset curves in terms of resultants. In the light of Theorem~\ref{Eps_Degree} their result says the following.

\begin{proposition}[Theorem $35$ from~\cite{SegundoSendra09}]
Let $X$ be a plane curve defined by the polynomial $f(x_1,x_2)$ of degree $d$. The ED degree of $X$ equals
\[
\mathrm{deg}_{x_1,x_2}\left(PP_{y_1,y_2}(\mathrm{Res}_{x_3}(F(x_H),N(x_H,y)))\right),
\]
where $x_H=(x_1,x_2,x_3)$, $y=(y_1,y_2)$, $F(x_H)$ is the homogenization of $f$ with respect to a new variable $x_3$, $N(x_H,y)=-F_2(x_H)(y_1x_3-x_1)+F_1(x_H)(y_2x_3-y_1)$, where $F_1$ and $F_2$ are the homogenized partial derivatives of $f$ and $PP_{y_1,y_2}$ denotes the primitive part of the given polynomial with respect to $\{y_1,y_2\}$. 
\end{proposition}

\begin{example}[\textbf{Determinantal varieties}]
Suppose $n \leq m$ and let $M_{n,m}^{\leq r}$ be the variety of $n\times m$ matrices over $\mathbb{C}$ of rank at most $r$. This variety is defined by the vanishing of all $(r+1)\times (r+1)$ minors of the matrix. For a fixed $\epsilon$ the construction of the offset hypersurface reduces to determining the set of matrices that have at least one critical rank $r$ approximation at squared distance $\epsilon^2$. By~\cite[Example 2.3]{DHOST16} all the critical rank $r$ approximations to a matrix $U$ are of the form \[T_1\cdot \mathrm{Diag}(0,0,...,\sigma_{i_1},...,\sigma_{i_r},0,...,0)\cdot T_2,\] where the singular value decomposition of $U$ is equal to $U=T_1\cdot\mathrm{Diag}(\sigma_1,...,\sigma_n)\cdot T_2$, with $\sigma_1>...>\sigma_n$ singular values and $T_1,T_2$ orthogonal matrices of size $n\times n$ and $m\times m$.  Now by~\cite[Corollary 2.3]{HelmkeS92} the squared distance of such a critical approximation from $U$ is exactly
\[
\sigma_{i_1}^2+\ldots+\sigma_{i_r}^2.
\]
Recall that $\sigma_i^2$ are the eigenvalues of $U\cdot U^T$, so what we seek is that the sum of an $r$-tuple of the eigenvalues of $U\cdot U^T$ equals $\epsilon^2$. Let us denote by $\bigwedge^{(r)}( U\cdot U^T)$ the $r$-th \textbf{additive compound matrix} of $U\cdot U^T$. For the construction of this object we refer to~\cite[P14]{Fiedler74}. The additive compound matrix is an ${n \choose r}\times {n \choose r}$ matrix with the property that its eigenvalues are the sums of $r$-tuples of eigenvalues of the original matrix~\cite[Theorem 2.1]{Fiedler74}. So the eigenvalues of $\bigwedge^{(r)}( U\cdot U^T)$ are exactly $\sigma_{i_1}^2+\ldots+\sigma_{i_r}^2$. Putting this together we get that the offset hypersurface of $M_{n\times m}^{\leq r}$ is defined by the vanishing of
\[
\det\left(\bigwedge^{(r)}( U\cdot U^T)-\epsilon^2\cdot I_{{n \choose r}}\right).
\]
Observe that the $\epsilon$ degree of this polynomial is $2\cdot {n \choose r}$, which is indeed two times the ED degree of $M_{n\times m}^{\leq r}$ (see \cite[Example 2.3]{DHOST16}).
\end{example}

\subsection{Offset discriminant}
We now consider the restriction
$\mathrm{pr}_y|_{\mathcal{OC}_{\epsilon}(X)}: \mathcal{OC}_{\epsilon}(X)\ \to\ \mathcal{O}_{\epsilon}(X).$ We claim that, for generic $\epsilon$, this restriction is one-to-one outside its branch locus. Indeed if we fix a generic $y_0\in\mathcal{O}_{\epsilon}(X)$, then the fiber above $y_0$ equals

\[\mathrm{pr}^{-1}_y(y_0)=\left(V(d(x,y)-\epsilon^2)\cap (\mathbb{C}\times\{y_0\})\right)\cap \left(\mathcal{E}(X)\cap (\mathbb{C}\times\{y_0\})\right).\]

By the definition of ED degree we have that 
\[
\mathcal{E}(X)\cap (\mathbb{C}\times\{y_0\})=\{(x_1,y_0),\ldots,(x_{EDdegree(X)},y_0)\}.
\]

Combining these we get, 
\[
\mathrm{pr}^{-1}_y(y_0)=\{(x_1,y_0),\ldots,(x_{EDdegree(X)},y_0)\}\cap \left(V(d(x,y)-\epsilon^2)\cap (\mathbb{C}\times\{y_0\})\right).
\]
This means that the fiber consists of pairs $(x,y_0)$, such that $x$ is a critical point of the squared distance function from $y_0$ and is of squared distance $\epsilon^2$ from $y_0$. For generic $y_0\in \mathcal{O}_{\epsilon}(X)$ and generic $\epsilon$, there is exactly one such critical point. Otherwise a $y_0$ with at least two elements in the fiber would be a doubly covered point of the offset hypersurface, hence part of its singular locus, which is of strictly lower dimension than the offset hypersurface itself. Indeed the branch locus of the restriction of $\mathrm{pr}_y$ is (generically) a hypersurface inside $\mathcal{O}_{\epsilon}(X)$ (by the Nagata-Zariski Purity theorem \cite{Purity2,Purity1}), hence a codimension two variety in $\mathbb{C}_y^n$, and it consists of points $y$ for which there exist at least two $x_1,x_2\in X_{reg}$, such that $(x_1,y),(x_2,y)\in \mathcal{OC}_{\epsilon}(X)$, or one $(x_1,y)\in \mathcal{OC}_{\epsilon}(X)$ with 
multiplicity greater than one. We denote the closure of the union of all branch loci, over $\epsilon$ in $\mathbb{C}$, by $B(X,X)$, which is the \textbf{bisector hypersurface of the variety $X$} (see for instance \cite{Bisector1,Bisector2}). Note that the variety itself is a component of $B(X,X)$ because for $\epsilon=0$ the variety is covered doubly under the projection $\mathrm{pr}_y$. We call the set of doubly covered points such that $(x_1,y)\neq(x_2,y)\in\mathcal{OC}_{\epsilon}(x)$, the \textbf{proper bisector locus}, and we denote it by $B_0(X,X)$\label{Prop_bisector_def}. 
In summary, we have the following result.
\begin{proposition}\label{one_to_one}
For a fixed generic $\epsilon$ the projection $
\mathrm{pr}_y|_{\mathcal{OC}_{\epsilon}(X)}: \mathcal{OC}_{\epsilon}(X)\ \to\ \mathcal{O}_{\epsilon}(X)$
is one-to-one outside the bisector hypersurface $B(X,X)$.
\end{proposition}

Let us see how this relates to (not the union but) the collection of offset hypersurfaces for all $\epsilon$. This collection is the offset family, $\mathcal{O}(X)$, and it is a hypersurface in $\mathbb{C}^n_{y}\times \mathbb{C}^1_{\epsilon}$. Its defining polynomial is the same as of $\mathcal{O}_{\epsilon}(X)$. Let us denote this polynomial by $f(y,\epsilon)$. Now if we consider $f(y,\epsilon)$ to be a univariate polynomial in the variable $\epsilon$, then we can compute its discriminant $\mathrm{Discr}_{\epsilon}(f)$, which is a polynomial in the variables  $y$, with the property that $f(y_0,\epsilon)$ has a double root (in $\epsilon$) if and only if $y_0$ is in the zero set of  $\mathrm{Discr}_{\epsilon}(f)$.

Now $y_0\in \mathrm{Discr}_{\epsilon}(f)$ if and only if there are fewer than two times ED degree many distinct roots, not counting multiplicities, of $f(y_0,\epsilon)$. By Theorem~\ref{Eps_Degree} this means that either $y_0$ has a non-generic number of critical points, meaning that $y_0$ is an element of the ED discriminant (for definition recall \ref{def_Ed_disc}), or there are two critical points $x_i\neq x_j$, such that
\[
d(x_i,y_0)=d(x_j,y_0),
\]
meaning that the projection $\mathrm{pr}_y: \mathcal{OC}_{\epsilon}(X)\to\mathcal{O}_{\epsilon}(X)$ is not one-to-one over $y_0$, so $y_0$ in an element of the branch locus. To summarize this we have the following proposition.

\begin{proposition}
Suppose that $\mathcal{O}(X)$ is defined by the vanishing of $f(y,\epsilon)$. Then the zero locus of the $\epsilon$-discriminant of $f$ is the union of the ED discriminant of $X$ and the bisector hypersurface of $X$. So we have that 
\[
\mathrm{Discr}_{\epsilon}(f)=\Sigma(X)\cup B(X,X).
\]
\end{proposition}

Throughout the rest of the article we call the union of the ED discriminant and the bisector hypersurface the \textbf{offset discriminant}, denoted $\Delta(X)$. And we recall that by construction, it is \textit{the envelope of all the offset hypersurfaces} to $X$.

\begin{example}[\bf{Offset discriminant of an ellipse}]
Let $X$ be the ellipse defined by $x_1^2+4x_2^2-4$. The offset family of the ellipse is defined by the vanishing of the polynomial from Example~\ref{ellipse_offset_exa}. The $\epsilon$-discriminant of this polynomial factors into five irreducible components. One of them is the defining polynomial of the sextic {\em Lam\'e curve} 
\[
64y_1^6+48y_1^4y_2^2+12y_1^2y_2^4+y_2^6-432y_1^4+756y_1^2y_2^2-27y_2^4+972y_1^2+243y_2^2-729,
\]
with zero locus $\Sigma(X)$, the ED discriminant (evolute) of $X$. The remaining four components comprise the bisector curve $B(X,X)$ of the ellipse. Two out of these four components of $B(X,X)$ are the $x$- and $y$- axes (the proper bisector locus $B_0(X,X)$), one of the components is the ellipse itself (because for $\epsilon=0$ the variety is doubly covered under the projection $\mathrm{pr}_y$) and the remaining component is fully imaginary. A cartoon of the real part of $\Delta(X)$ can be seen in Figure~\ref{ellipse_disc}. The ellipse is black, the proper bisector locus (the axis) is blue and the ED discriminant is red.

\begin{figure}[h]
\begin{center}
\vskip -0.1cm
\includegraphics[scale=0.6]{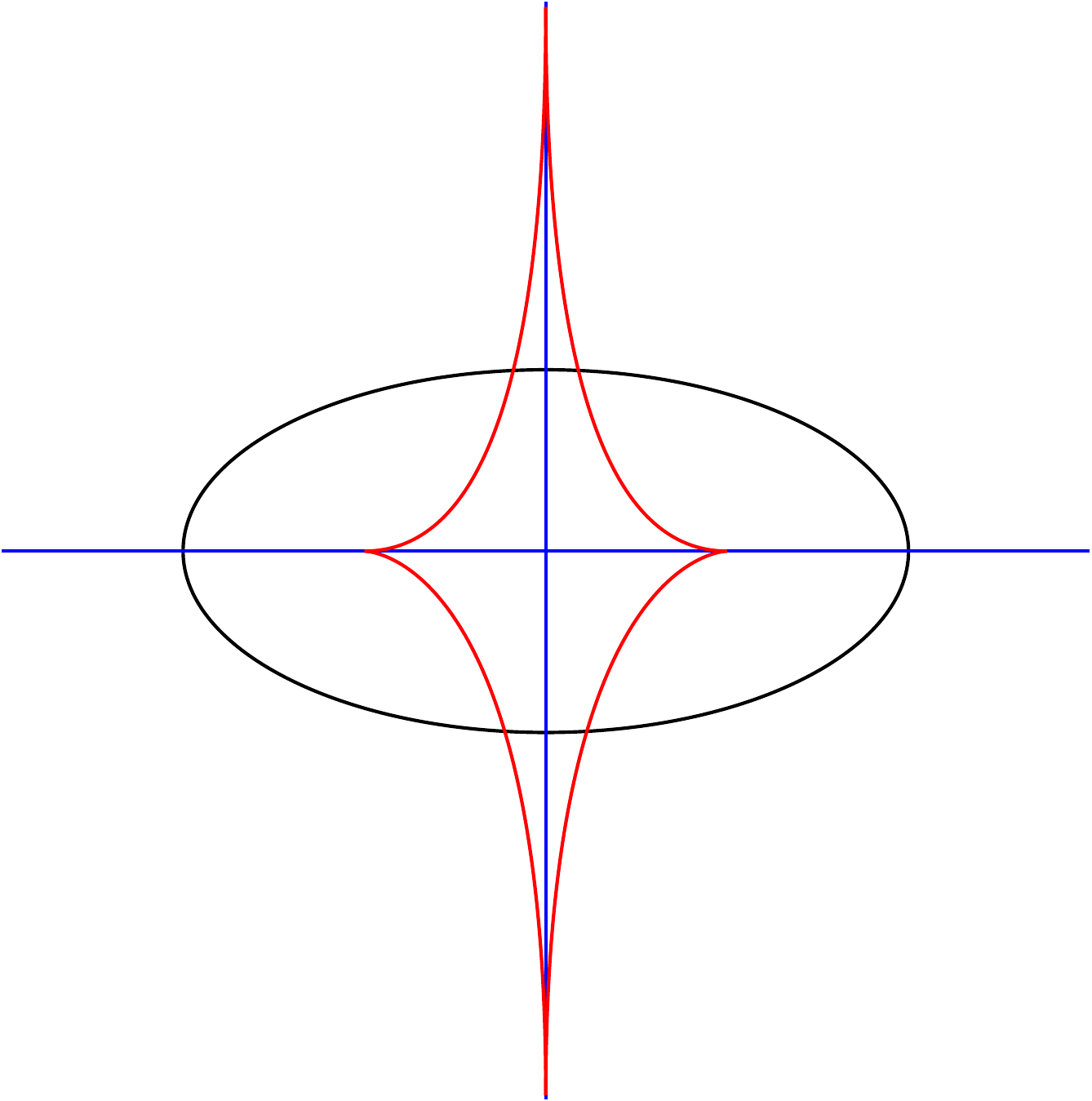}
\vskip -0.4cm
\caption{The ED discriminant and the bisector curve of the ellipse.}
\label{ellipse_disc}
\end{center}
\end{figure}

\end{example}

\begin{corollary}\label{degree_discriminant}
Let $X$ be an irreducible variety in $\mathbb{C}^n$. The degree of its offset discriminant $\Delta(X)$ (hence also the degree of its ED discriminant $\Sigma(X)$ and the degree of the bisector hypersurface $B(X,X)$) is bounded from above by
\[
2\cdot \mathrm{deg}_y(\mathcal{O}(X))\cdot\left(4\cdot\mathrm{EDdegree}(X)-2\right).
\]
\end{corollary}

\begin{proof}
Suppose that the offset family $\mathcal{O}(X)$ is the zero set of the polynomial $f(y,\epsilon)$.
The offset discriminant $\Delta(X)$ is the discriminant of the univariate polynomial \[f(y,\epsilon)=a_0(y)+\ldots+a_{d-1}(y)\cdot \epsilon^{d-1}+a_{d}(y)\epsilon^d\] in the variable $\epsilon$ of degree $d=\mathrm{deg}_{\epsilon}(f)$.  So $\Delta(X)$ is a homogeneous polynomial in the coefficients $a_0(y),\ldots,a_d(y)$ of degree equal to $2\cdot d-2$. By Theorem~\ref{Eps_Degree} we have that $d=2\cdot \mathrm{EDdegree}(X)$. Now because the discriminant is a homogeneous polynomial in the coefficients we get the desired degree bound. 
\end{proof}

\begin{example}[\textbf{Degree bounds of the offset discriminant}]
The following table contains degree bounds of the offset discriminant based on the formula above and the total degree formulae, $\deg_y(\mathcal{O}(X))$, by San Segundo and Sendra~\cite[Appendix. Table of offset degrees]{SegundoSendra09}.

\begin{center}
 \begin{tabular}{|c|c| c| c| c|} 
 \hline
Name of $X$ & Defining poly. of $X$ & $\mathrm{deg}_y\mathcal{O}(X)$ & $\mathrm{deg}_{\epsilon}\mathcal{O}(X)$ & $\mathrm{deg}_y\Delta(X) \leq $ \\ [0.5ex] 
 \hline\hline
 Circle& $x_1^2+x_2^2-1$ & 4 & 4 & 24 \\ 
 \hline
 Parabola& $x_2-x_1^2$ & 6 & 6 & 60 \\
 \hline
Ellipse& $x_1^2+4x_2^2-4$ & 8 & 8 & 112 \\
 \hline
 Cardioid&$(x_1^2+x_2^2+x_1)^2-x_1^2-x_2^2$ & 10 & 8 & 140 \\
 \hline
 Rose($3$ petals)&$(x_1^2+x_2^2)^2+x_1(3x_2^2-x_1^2)$ & 14 & 12 & 308 \\ [1ex] 
 \hline
\end{tabular}
\end{center}
\end{example}

As the degree of an algebraic variety is a proxy for computational complexity, these degree bounds serve a reminder of the challenges of computing offsets, and thus persistence. We also see how the difficulty depends on the nature of the starting variety $X$. 

\section{Algebraicity of persistent homology}\label{Sec3}

As an application of this knowledge of offset hypersurfaces, we study the persistent homology of the offset filtration of an algebraic variety. We define this to be the homology of the set of points within distance $\epsilon$ of the variety, which is bounded by the offset hypersurface. First, we review background material on persistent homology. Next, we define the persistent homology of the offset filtration in terms of the offset hypersurface. We prove the algebraicity of two quantities involved in computing persistent homology. We do not present a new algorithm to compute persistent homology. However, we do provide theoretical foundations to show that it is possible for an algorithm to compute persistent homology barcodes exactly. If the expected output of a computation is algebraic over the rational numbers, this means it can be computed using polynomials of finite degree, and thus it is possible for the algorithm to terminate. We also discuss the relevance of the offset discriminant to persistent homology.

\subsection{Background on Persistent Homology}
We now provide an abbreviated introduction to persistent homology. For further background, we refer the reader to \cite{Carlsson}. 

The persistent homology of a finite subset of $\mathbb{R}^n$ at parameter $\epsilon$ is defined as the homology of a simplicial complex, called the \v{Cech} complex, associated to a covering of the point cloud by hyperballs of radius $\epsilon$. By the nerve theorem, the \v{C}ech complex has the same homology as the covering.  

\begin{definition}
Let $X \subset \mathbb{R}^n$ and $\epsilon>0$ a parameter. Let $\sigma$ be a finite subset of $X$. The \textbf{\v{C}ech complex} of $X$ at radius $\epsilon$ is
\[\displaystyle{C_X(\epsilon)=\left\{ \sigma \subset X \text{ s.t. } \bigcap_{x \in \sigma} B_\epsilon(x) \neq 0\right\},}\]
an abstract simplicial complex where the $n$-faces are the subsets of size $n$ of $X$ with nonempty $n$-wise intersection.
\end{definition}

From these simplicial complexes, we obtain a filtration for which we can define persistent homology. 

Following \cite{EH}, consider a simplicial complex, $K$, and a function $f: K \to \mathbb{R}$. We require that $f$ be \textit{monotonic} by which we mean it is non-decreasing along chains of faces, that is, $f(\sigma) \leq f(\tau)$ whenever $\sigma$ is a face of $\tau$. Monotonicity implies that the sublevel set, $K(a)=f^{-1}(-\infty, a],$ is a subcomplex of $K$ for every $a \in \mathbb{R}$. Letting $m$ be the number of simplices in $K$, we get $n+1 \leq m+1$ different subcomplexes, which we arrange as an increasing sequence, 
\[ \emptyset=K_0 \subseteq K_1 \subseteq \dots \subseteq K_n=K.  \]
In other words, if $a_1<a_2< \dots <a_n$ are the function values of the simplices in $K$ and $a_0= - \infty$ then $K_i=K(a_i)$ for each $i$. We call this sequence of complexes the \textit{filtration}
of $f$. 

For every $i \leq j$ we have an inclusion map from the underlying space of $K_i$ to that of $K_j$ and therefore an induced homomorphism, $f_q^{i,j}: H_q(K_i) \to H_q(K_j),$ for each dimension $q$. 

\begin{definition}
The \textbf{q-th persistent homology groups} are the images of the homomorphisms induced by inclusion, $H_q^{i,j}=\text{im }f_q^{i,j},$ for $0 \leq i \leq j \leq n$. The corresponding \textbf{q-th persistent Betti numbers} are the ranks of these groups, $\beta_q^{i,j}= \text{rank }H_q^{i,j}$. 
\end{definition}

As a consequence of the Structure Theorem for PIDs, the family of modules $H_q(K_i)$ and homomorphisms $f_q^{i,j}: H_q(K_i) \to H_q(K_j)$ over a field $F$ yields a decomposition 

\begin{equation}\label{pmod}
\displaystyle{H_{q} (K_i;F) \cong \bigoplus_i x^{t_i} \times F[x] \bigoplus \left( \bigoplus_j x^{r_j} \cdot \left( F[x] / (x^{s_j} \cdot F[x]) \right)   \right),} \end{equation} 
where $t_i,r_j,$ and $s_j$ are values of the persistence parameter $\epsilon$ \cite{Ghrist}. 

The free portions of Equation \ref{pmod} are in bijective correspondence with those homology generators which appear at parameter $t_i$ and persist for all $\epsilon>t_i$, while the torsional elements correspond to those homology generators which appear at parameter $r_j$ and disappear at parameter $r_j+s_j$.

To encode the information given by this decomposition, we create a graphical representation of the $q$-th persistent homology group called a \textbf{barcode} \cite{Ghrist}. For each parameter interval $[r_j,r_j+s_j]$ corresponding to a homology generator, there is a horizontal line segment (bar), arbitrarily ordered along a vertical axis. The persistent Betti number $\beta_q^{i,j}$ equals the number of intervals in the barcode of $H_q(K_i;F)$ spanning the parameter interval $[i,j]$. 

Persistent homology is defined using the \v{C}ech complex, but it is hard to compute using the \v{C}ech complex because this requires storing simplices is many dimensions. In practice, persistent homology is often computed using the Vietoris-Rips complex, a simplicial complex determined entirely by its edge information. The Vietoris-Rips complex is defined as follows. 

\begin{definition}
Let $X \subset \mathbb{R}^n$ and $\epsilon>0$ a parameter. Let $\sigma$ be a finite subset of $X$.  The \textbf{Vietoris-Rips complex} of $X$ at radius $\epsilon$ is
\[\displaystyle{VR_X(\epsilon)=\left\{ \sigma \subset X \text{ s.t. } B_\epsilon(x) \cap B_\epsilon(y) \neq 0 \text{ for all pairs} (x,y)
 \in \sigma \right\},}\]
an abstract simplicial complex where the $n$-faces are the subsets of size $n$ of $X$ such that every pair of points in the subset has nonempty pairwise intersection.
\end{definition}

Using Jung's theorem, one can show that $C_X(\epsilon) \subseteq VR_X( \sqrt{2} \epsilon) \subset C_X( \sqrt{2} \epsilon)$, so that the Vietoris-Rips complex can indeed be used to approximate persistent homology \cite{Ghrist}. 

We include here an example of the real variety defined by the Trott curve and a barcode representing its persistent homology, computed by taking a sample of points on the variety.

\begin{example}[\textbf{The barcodes of the Trott curve}]
In dimension 1, the first four bars correspond to the cycles in each of the four components of the real variety. As epsilon increases, these cycles fill in, and then the components join together in one large circle. This demonstrates how persistent homology can detect the global arrangement of the components of a variety. The barcodes were computed using Ripser, which uses the Vietoris-Rips complex \cite{Ripser}. 

 \begin{figure}[h!]
   \begin{center}
    \includegraphics[scale=0.7]{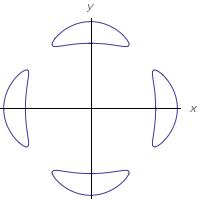}\label{2}
  \end{center}
  \caption{The Trott curve.}
  
      \begin{center}
    \includegraphics[scale=0.4]{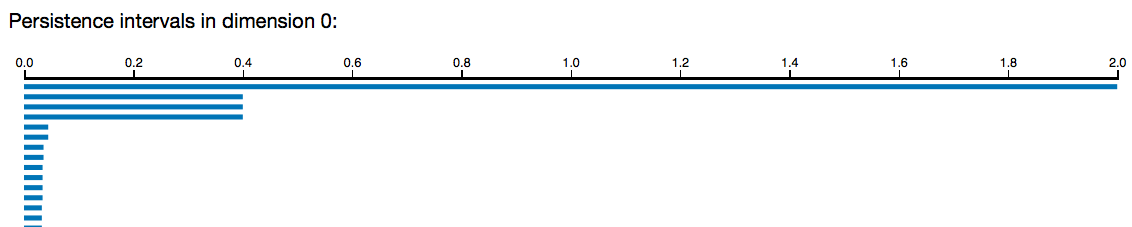}
  \end{center}
  
       \begin{center}
    \includegraphics[scale=0.4]{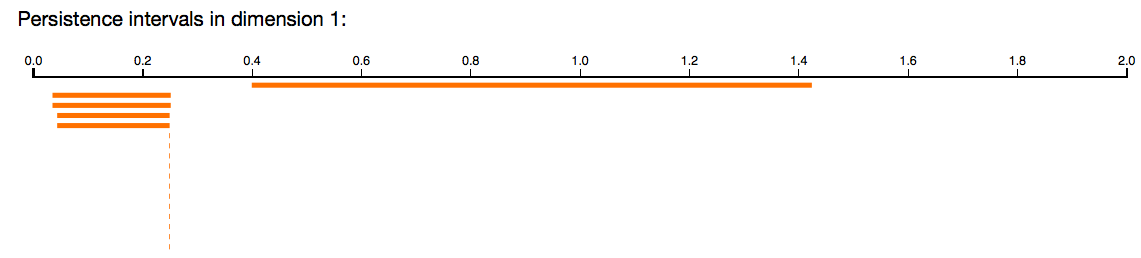}
  \end{center}
  \caption{Barcodes for the Trott curve in homological dimensions 0 and 1.}
  
  \end{figure}

\end{example}

\subsection{Persistent homology of the offset filtration of a variety}

Persistent homology is typically defined for a finite metric space. To compute the persistent homology of a variety $X$, one might sample a finite set of points from the variety and compute the \v{C}ech complex of those points. The equivalent of the \v{C}ech complex $C_X(\epsilon)$  obtained from sampling every point on the variety would be the set of all points within $\epsilon$ of the variety. For this reason, we define the \textbf{persistent homology of the offset filtration of a variety $X$ at parameter $\epsilon$} as the homology of the subset 
\[ X_\epsilon=\{ x | \text{ there exists } y \in V \text{ with } ||x-y|| \leq \epsilon \} \subset \mathbb{R}^n \] consisting of all points within $\epsilon$ of the variety.

Since the $\epsilon$-offset hypersurface is the envelope of a family of $\epsilon$-hyperballs centered on the variety, we can define the the persistent homology of the offset filtration of a variety at parameter $\epsilon$ equivalently as the homology of the set bounded by $\mathcal{O}_{\epsilon}(X)$.

To define barcodes with respect to this filtration, we use Hardt's theorem from real algebraic geometry. We now make the necessary definitions and state the theorem. 

\begin{definition}[Definition $9.3.1$ from~\cite{BCR}]
Let $S, T$ and $T^{'}$ be semi-algebraic sets, $T^{'} \subset T$, and let $f: S \to T$ be a continuous semi-algebraic mapping. A \textbf{semi-algebraic trivialization of} $f$ \textbf{over} $T^{'}$, with fiber $F$, is a semi-algebraic homeomorphism $\theta: T^{'} \times F \to f^{-1}(T^{'})$, such that $f \circ \theta$ is the projection mapping $T^{'} \times F \to T^{'}$. We say that the semi-algebraic trivialization $\theta$ is \textbf{compatible with a subset} $S^{'}$ \textbf{of} $S$ if there is a subset $F^{'}$ of $F$ such that $\theta (T^{'} \times F^{'})= S^{'} \cap f^{-1}(T^{'})$. 
\end{definition}

\begin{lemma}[Hardt's Theorem, $9.3.1$ from~\cite{BCR}]
Let $S$ and $T$ be two semi-algebraic sets, $f: S \to T$ a continuous semi-algebraic mapping, $(S_j)_{j=1, \dots, q}$ a finite family of semi-algebraic subsets of $S$. There exist a finite partition of $T$ into semi-algebraic sets $T= \cup_{l=1}^r T_l$ and, for each $l$, a semi-algebraic trivialization $\theta_l: T_l \times F_l \to f^{-1}(T_l)$ of $f$ over $T_l$, compatible with $S_j$, for $j=1, \dots, q$. 
\end{lemma}

Let $S=\{(X_{\epsilon}, \epsilon) |\epsilon \in [0, \infty )\} \subset \mathbb{R}^{n+1}$ and let $\mathrm{pr}_{\epsilon}: S \to \mathbb{R}$ be the projection to $\epsilon$. By Hardt's theorem,  there is a partition of $\mathbb{R}$ into finitely many intervals $I_l=[\delta_l,\epsilon_l]$ for $ l \in \{1, \dots, j\}$ such that the fibers $\mathrm{pr}_{\epsilon}^{-1}(\epsilon)=X_{\epsilon}$ for all $\epsilon \in [\delta_l,\epsilon_l]$ are homeomorphic. Thus we can create the \textbf{offset filtration barcode} of $X$. 

We show that $\{\delta_l \} \cup \{\epsilon_l \}$ for $l \in \{1, \dots, j\}$, 
the values of the persistence parameter $\epsilon$ at which a bar in the offset filtration barcode appears or disappears, are algebraic over the field of definition of a real affine variety $X$. As a consequence, the persistent homology of the offset filtration of $X$ can be computed exactly.  

The proof relies on two lemmas from real algebraic geometry. We describe the content of these lemmas. The setting of the results is a real closed field, which we now define. 

\begin{definition}[Definitions $1.1.9$ and $1.2.1$ from~\cite{BCR}]
A field $R$ is a \textbf{real field} if it can be ordered. A real field $R$ is a \textbf{real closed field} if it has no nontrivial real algebraic extension. 
\end{definition}

The first lemma is Tarski-Seidenberg's Theorem, a fundamental result in real algebraic geometry which implies that quantifier elimination is possible over real closed fields. This means that for every system of polynomial equations and inequalities that can described using logical quantifiers, there is an equivalent system without the quantifiers. 

To state the result, we use the following notation, where $R$ is a real closed field and $a \in R$. 
\begin{align*}
\text{sign}(a)&=0         &  \text{if } &a=0 \\
\text{sign}(a)&=1        &  \text{if }&a > 0 \\
\text{sign}(a)&=-1   &  \text{if }& a<0         
\end{align*}
\begin{lemma}[Tarski-Seidenberg's Theorem, $1.4.2$ from~\cite{BCR}]\label{TST}
Let $f_i(X,Y)=h_{i,m_i}(Y)X^{m_i} + \dots + h_{i,0}(Y)$ for $i=1, \dots, s$ be a sequence of polynomials in $n+1$ variables with coefficients in $\mathbb{Z}$, where $Y=(Y_1, \dots, Y_n)$. Let $\epsilon$ be a function from $\{1, \dots, s\}$ to $\{ -1,0,1 \}$. Then there exists a boolean combination $\mathcal{B}(Y)$ (i.e. a finite composition of disjunctions, conjunctions and negations) of polynomial equations and inequalities in the variables $Y$ with coefficients in $\mathbb{Z}$ such that for every real closed field $R$ and for every $y \in R^n$, the system 
 
 \[
    \left\{
                \begin{array}{c}
                  \text{sign}(f_1(X,y))=\epsilon(1) \\
                   \vdots \\
                  \text{sign}(f_s(X,y))=\epsilon(s) 
                \end{array}
              \right.
  \]
has a solution $x$ in $R$ if and only if $\mathcal{B}(y)$ holds true in $R$. 
\end{lemma}

The second result gives an isomorphism of homology groups of a variety and its restriction to the closed subfield of real algebraic numbers over $\mathbb{Q}$.

\begin{lemma}[Theorem 4.2 from~\cite{Delfs}]
\label{delfslemma}
Let $R \subset \tilde{R}$ be an inclusion of real closed fields. Let $\tilde{X} \subset \tilde{R}^n$  be a semialgebraic set and $X=\tilde{X} \cap R^n$. Then there are canonical isomorphisms 
\[ H_q(X) \cong H_q(\tilde{X}) ,\] 
\[ H^q(X) \cong H^q(\tilde{X})  . \]
\end{lemma}

The proof spans the first four sections of \cite{Delfs} and the first-named author Delfs' thesis. Tarski-Seidenberg's theorem is used to establish a base extension functor from the category of semialgebraic maps and spaces over $R$ to the corresponding category for $\tilde{R}$. Using base extension, a triangulation of $\tilde{X} \subset \tilde{R}^n$ can be obtained from a triangulation of $X \subset R^n$. This establishes the desired isomorphism of homology groups. 

\begin{theorem} \label{algebraicity} (Algebraicity of persistent homology barcodes.)
Let $f_1,\dots,f_s$ be polynomials in $\mathbb{Q}[x_1,\dots,x_n]$ with $X_{\mathbb{R}}=V_\mathbb{R}(f_1,\dots,f_s)$. Then the values of the persistence parameter $\epsilon$ at which a bar in the offset filtration barcode appears or disappears are real numbers algebraic over $\mathbb{Q}$.
\end{theorem}

\begin{proof}
Let $\mathbb{R}_{alg}$ denote the closed subfield of real algebraic numbers over $\mathbb{Q}$.
Let $S_{alg}= \{(X_{\epsilon} \cap \mathbb{R}_{alg}^n, \epsilon) |\epsilon \in [0, \infty) \}$. Then $S_{alg}$ is a semialgebraic subset of $\mathbb{R}_{alg}^{n+1}$ in the sense of \cite[Definition 2.4.1]{BCR} since $S_{alg}$ is defined by polynomial equalities and inequalities with coefficients in $\mathbb{Q}$. Let $\mathrm{pr}_{\epsilon}: S_{alg} \to \mathbb{R}_{alg}$ be the projection to $\epsilon$. 

Since $\mathbb{R}_{alg}$ is a closed subfield of $\mathbb{R}$, by Hardt's theorem there is a partition of $\mathbb{R}_{alg}$ into finitely many sets $I_l=[\delta_l,\epsilon_l] \cap \mathbb{R}_{alg}$ with $\delta_l,\epsilon_l \in \mathbb{R}_{alg}$ for $l \in \{1, \dots, j  \}$ such that $\mathrm{pr}_{\epsilon}^{-1}(\epsilon)=X_{\epsilon} \cap \mathbb{R}_{alg}^n$ for all $\epsilon \in [\delta_l,\epsilon_l]$ are homeomorphic. 

By Lemma \ref{delfslemma}, there is an isomorphism of homology groups 
\[ H_q(X_{\epsilon} \cap \mathbb{R}_{alg}^n) \cong H_q(X_{\epsilon})\]
for each $q$ and $\epsilon$. So the partition by the  sets $I_l=[a_l,b_l] \cap \mathbb{R}_{alg}$ given by Hardt's theorem corresponds to a partition of $\mathbb{R}$ by intervals $\tilde{I_l}=[\delta_l,\epsilon_l] \subset \mathbb{R}$ such that $X_{\epsilon}$ for all $\epsilon \in [\delta_l,\epsilon_l]$ with $\delta_l,\epsilon_l \in \mathbb{R}_{alg}$ are homeomorphic.   Thus $ \{\delta_l \}_ {l \in \{1, \dots, j\}} \cup \{\epsilon_l \}_ {l \in \{1, \dots, j\}} \subset \mathbb{R}_{alg}$. 
\end{proof}

\subsection{Using the offset discriminant to identify points of interest for persistent homology}
\FloatBarrier

We now discuss the bisector hypersurface (a component of the offset discriminant) in the context of persistent homology of the offset filtration. We first show how the bisector hypersurface can help identify points where homological events occur. Then we discuss the medial axis, a subset  of the proper bisector locus of $X$, which gives information about the density of sampling required to compute the persistent homology accurately. 

Consider a bar in the offset filtration barcode corresponding to the top dimension Betti number. To each such bar, there corresponds a $y \in \Delta(X)$. Informally, this $y$ is the center of the $n$-dimensional hole corresponding to the bar. We illustrate with the example of the circle $x_1^2+x_2^2=r^2 \subset \mathbb{R}^2$ in Figure \ref{Offset hypersurface construction}. The persistent homology of the circle has $\beta_1=1$ for all $\epsilon<r$, and a real component of the offset hypersurface is a smaller circle inside the circle. When $\epsilon=r$, the offset hypersurface is simply the point at the center of the circle, and $\beta_1=0$. 

\begin{figure}[h]
\begin{center}
\vskip -0.2cm
\includegraphics[scale=0.7]{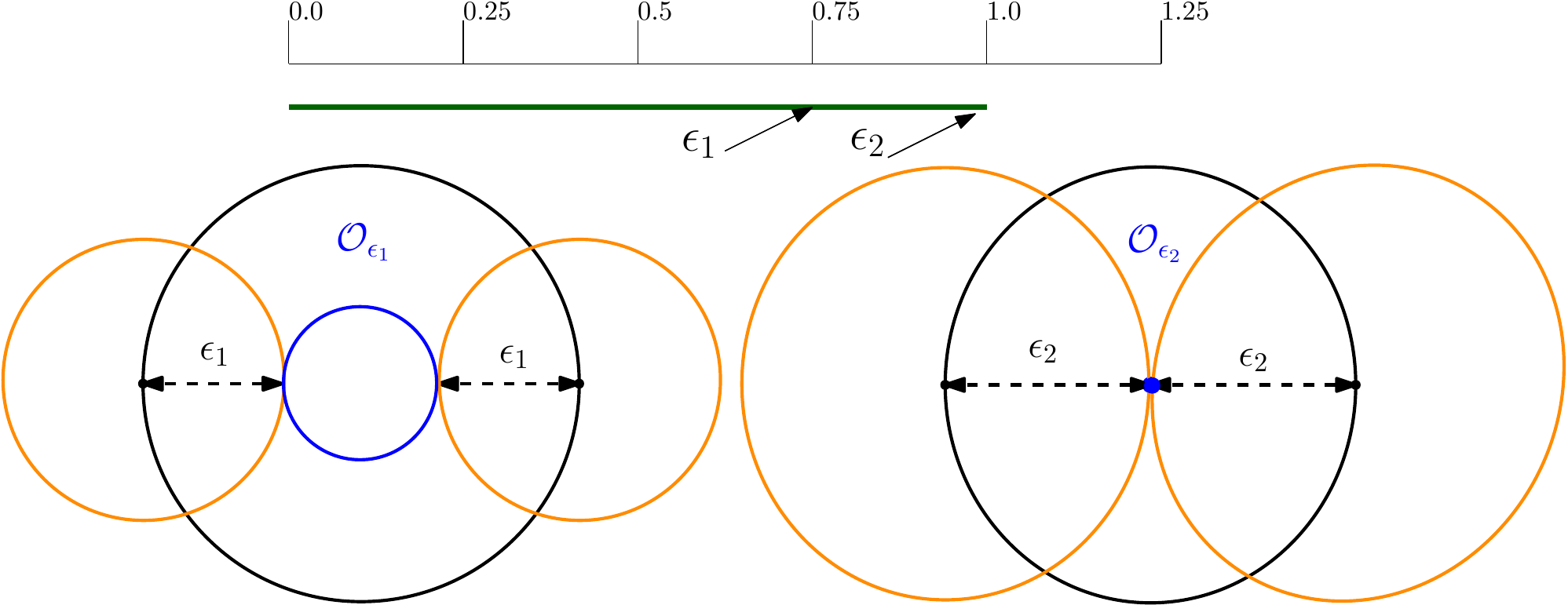}
\vskip -0.2cm
\caption{These pictures illustrate how the offset variety provides a geometric interpretation of the endpoints of a bar. The black circle is the variety $X$ and the orange circles are $\epsilon$-balls around $X$. When $\epsilon$ reaches the radius of the black circle, the blue offset hypersurface $\mathcal{O}_{\epsilon}$ has an isolated real point.}
\label{Offset hypersurface construction}
\end{center}
\end{figure}


\begin{theorem}\label{interesting_points} (Geometric interpretation of endpoints in barcode.) Let $X \subset \mathbb{R}^{n+1}$ be a hypersurface. Let $J= \{ [\delta_l,\epsilon_l] | l \in \{1, \dots,m \}  \}$ be the set of intervals  in the barcode for the top dimensional Betti number $\beta_n$. Then each interval endpoint $\epsilon_l$ corresponds to a point $y_l \in  \mathcal{O}_{\epsilon_l}(X)$ on the bisector hypersurface $B(X,X)$ such that $y_l$ is the limit of a sequence of centers of hyperballs contained in the complement of $\mathcal{O}_{\epsilon}(X)$ as $\epsilon \to \epsilon_l$. 
\end{theorem}

We make the following observations. First, the correspondence does not assign each interval to a unique point on the offset discriminant. Consider the persistent homology of the Trott curve. In dimension 1, there is one interval corresponding to four cycles, and we do not specify to which cycle to assign the interval. Second, we note that the set of $y_l$ corresponding to endpoint intervals may not be $0$-dimensional. For example, let $X$ be the torus. Then the set of $y_l$ contains a  circle. 

We also comment on the topology of the real algebraic varieties involved. Suppose $X \subset \mathbb{R}^2$ is a curve. Then $\mathcal{O}_{\epsilon}$ for $\epsilon \in (\epsilon_l-\epsilon^{'}, \epsilon_l)$ will have an oval component not present in $\mathcal{O}_{\epsilon}$ for $\epsilon>\epsilon_l$, so $y_l$ is an isolated real point of $\mathcal{O}_{\epsilon_l}$. 

\begin{proof}
 Fix $[\delta_l,\epsilon_l] \in J$.   Then there exists $\epsilon^{'}>0$ such that $\beta_n(X_{\epsilon_l})<\beta_n(X_{\epsilon})$ for all $\epsilon \in (\epsilon_l-\epsilon^{'}, \epsilon_l)$.

Since $\beta_n(X_{\epsilon_l})<\beta_n(X_{\epsilon})$ for all $\epsilon \in (\epsilon_l-\epsilon^{'}, \epsilon_l)$, there is some $\epsilon_1 \in (\epsilon_l-\epsilon^{'}, \epsilon_l)$ such that there is a maximum $\delta_1>0$ such that there is a ball $B_{\delta_1}^{n+1}(z_{\epsilon_1})$ such that $B_{\delta_1}^{n+1}(z_{\epsilon_1}) \subset X_{\epsilon_l}$ but $B_{\delta}^{n+1}(z_{\epsilon})$ is contained in a bounded connected component of $\mathbb{R}^{n+1} \setminus X_{\epsilon_1}$. 
Furthermore, since $X_{\epsilon_n} \supset  X_{\epsilon_{n-1}}$ for $\epsilon_{n-1} < \epsilon_n$,   there is a monotonically increasing sequence of $\{ \epsilon_n \}_{n=1, 2, \dots}$ with $\epsilon_1 < \epsilon_{n-1} < \epsilon_n < \epsilon_l$ such that for each $\epsilon_n$, there is a maximum $\delta_n$ such that there is a ball $B_{\delta_n}^{n+1}(z_{\epsilon_n}) \subset B_{\delta_{n-1}}^{n+1}(z_{\epsilon_{n-1}})$ with $B_{\delta_n}^{n+1}(z_{\epsilon_n})$ contained in a bounded connected component of $\mathbb{R}^{n+1} \setminus X_{\epsilon}$. 

Let $\epsilon_n \to \epsilon_l$. The diameter of the bounded connected component of  $X_{\epsilon_n}$ containing $z_{\epsilon_n}$ is less than that of $X_{\epsilon_{n-1}}$ containing $z_{\epsilon_{n-1}}$ for $\epsilon_{n-1}<\epsilon_n$ and $B_{\delta_1}^{n+1}(z_{\epsilon_1}) \subset X_{\epsilon_l}$, so $\delta_n \to 0$. So $\{ z_{\epsilon_n}\}$ is a Cauchy sequence in $\mathbb{R}^{n+1}$, so it converges. Let $\displaystyle{y=\lim_{\epsilon_n \to \epsilon_l} z_{\epsilon_n}}$.

Since $\delta_n$ is the maximum radius of such a ball, $B_{\delta_n}^{n+1}(z_{\epsilon_n}) \cap \mathcal{O}_{\epsilon_n}$ contains at least two points $\{y_{1,\epsilon_n}, y_{2, \epsilon_n} \}$. Corresponding to these points in $\mathcal{O}_{\epsilon_n}$ are at least two points in the offset correspondence $\mathcal{O}\mathcal{C}_{\epsilon_n}(X)$, say $\{(x_{1, \epsilon_n}, y_{1, \epsilon_n}), (x_{2, \epsilon_n}, y_{2, \epsilon_n}) \}$. 

As $\delta_n \to 0$, we have $||y_{1, \epsilon_n}- y_{2, \epsilon_n}||\to 0$ since $||y_{1, \epsilon_n}- y_{2, \epsilon_n}||\leq \delta_n$ . Thus $y \in B(X,X)$.
\end{proof}

Since the construction in the proof of the theorem is based on the limit of a converging sequence, the method above does not point to a new algorithm for computing barcodes or determining the reach. However, it shows how persistent homology barcodes can be studied in the algebraic geometry context of the bisector hypersurface.

\FloatBarrier
\subsection{Algebraicity of the reach}

The bisector hypersurface has further relevance to persistent homology because one of its components is the closure of the medial axis. The shortest distance from a manifold to its medial axis is called the reach. The reach of a manifold is a very important quantity in the computation of its persistent homology as it determines the density of sample points required to obtain the correct homology. We now define the reach, describe its importance in the theory of persistent homology, and prove its algebraicity. 

\begin{definition}\label{medial_axis}
Let $X$ be a real algebraic manifold in $\mathbb{R}^n$. The \textbf{medial axis} of $X$ is the set $M_X$ of all points $u \in \mathbb{R}^n$ such that the minimum Euclidean distance from $X$ to $u$ is attained by at least two distinct points in $X$. The \textbf{reach} $\tau(X)$ is the shortest distance between any point in the manifold $X$ and any point in its medial axis $M_X$. 
\end{definition}

Observe that the closure of the medial axis is by our definition the proper bisector locus (recall \ref{Prop_bisector_def}), hence the degree bound of the offset discriminant (see \ref{degree_discriminant}) gives upper bound on the degree of the closure of the medial axis as well.

We now state the theorem showing that sampling density depends on the reach. In particular, the smaller the reach (and thus, the curvier the manifold), the higher the density of sample points required to compute persistent homology accurately. We have adapted this from \cite{NSW}, where it is stated in terms of the reciprocal of a reach, a quantity which they call the condition number of the manifold. 

\begin{theorem}[Theorem 3.1 from~\cite{NSW}]\label{niyogi}
 Let $M$ be a compact submanifold of $\mathbb{R}^N$ of dimension $k$ with reach $\tau$. Let $\bar{x}= \{x_1, \dots, x_n \}$ be a set of $n$ points drawn in i.i.d fashion according to the uniform probability measure on $M$. Let $0< \epsilon < \frac{1}{2\tau}$. Let $\displaystyle{U=\bigcup_{x \in \bar{x}} B_{\epsilon}(x)}$ be a corresponding random open subset of $\mathbb{R}^N$. Let $\beta_1= \frac{vol(M)}{(\cos^k(\theta_1))vol(B_{\epsilon/4}^k)}$ and $\beta_2= \frac{vol(M)}{(\cos^k(\theta_2))vol(B_{\epsilon/8}^k)}$ where $\theta_1=\arcsin(\frac{\epsilon \tau}{8})$ and $\theta_2=\arcsin(\frac{\epsilon \tau}{16})$.  Then for all 
\[ n > \beta_1 \left(\log(\beta_2)+\log\left(\frac{1}{\delta}\right)\right)    \]
the homology of $U$ equals the homology of $M$ with high confidence (probability $> 1-\delta$). 
\end{theorem}

Studying a formulation for reach given by Federer with the tools of real algebraic geometry, we show the algebraicity of the reach. 

\begin{proposition}\label{reach}(Algebraicity of reach). Let $X$ be a real algebraic manifold in $\mathbb{R}^n$. Let $f_1,\dots,f_s \in \mathbb{Q}[x_1,\dots,x_n]$ with $X_{\mathbb{R}}=V_\mathbb{R}(f_1,\dots,f_s)$. Then the reach of $X$ is an algebraic number over $\mathbb{Q}$. 
\end{proposition}

\begin{proof}
Federer \cite[Theorem 4.18]{Fed} gives a formula for the reach $\tau$ of a manifold $X$ in terms of points and their tangent spaces: 
\begin{equation}
\label{eq:federer}
\qquad \tau(X)\,\,= \inf_{v \neq u \in X} \frac{||u-v||^2}{2\delta}, \quad \text{ where }\,\,
 \delta \, = \min_{x\in \mathrm{T}_vX} \Vert (u-v) - x\Vert. 
\end{equation}
Equation \ref{eq:federer} gives the following system of polynomial inequalities with rational coefficients

 \[
    \left\{
                \begin{array}{ll}
                   x\in \mathrm{T}_vX \\
                  \delta^2= \Vert (u-v) - x\Vert^2 \\
                  \delta>0 \\
                  2\delta\tau= ||u-v||^2 
                \end{array}
              \right.
  \]
in unknowns $\{ x, \delta, \tau, u, v   \}$. 
This defines a semialgebraic set in $\mathbb{R}^{3n+2}$ in the sense of \cite[Definition 2.4.1]{BCR}. Consider the projection on to $\mathbb{R}^2$ with coordinates $(\delta, \tau)$. 
By Tarski-Seidenberg's theorem (\cite[Theorem~1.4.2]{BCR}), the image is a semialgebraic set $S$. 

Project $S$ onto $\mathbb{R}$ with coordinate $\delta$. 
The minimum $\delta_0$ is attained in the closure of the image. 
Then $\overline{S} \cap \{ \delta=\delta_0 \} \subset \mathbb{R}$  is semialgebraic over $\mathbb{Q}$. It is bounded below by $0$. The reach is the infimum of this set, and thus is an algebraic number over $\mathbb{Q}$. 
\end{proof}

\noindent
{\bf Acknowledgements.} We thank Sara Kali\v snik and Kristin Shaw for helpful discussions and Bernd Sturmfels for initiating the project and offering essential insight. We also thank the anonymous reviewers for ideas on improving the article. This material is based upon work supported by the National Science Foundation Graduate Research Fellowship Program under Grant No. DGE 1752814. Any opinions, findings, and conclusions or recommendations expressed in this material are those of the authors and do not necessarily reflect the views of the National Science Foundation. E.H. was supported by the project \textit{Critical points from analysis to algebra} of Sapientia Foundation\--Institute for Scientific Research. Both authors are grateful to the Max Planck Institute for Mathematics in the Sciences in Leipzig for hosting and supporting them while they carried out this project.

\footnotesize {\bf Authors' addresses:}

\smallskip

\noindent Emil Horobe\c{t}, Sapientia Hungarian University of Transylvania \ 
\hfill {\tt horobetemil@ms.sapientia.ro}

\noindent Madeleine Weinstein,
UC Berkeley \hfill {\tt maddie@math.berkeley.edu}

\end{document}